\documentclass[review,12pt]{elsarticle}
\usepackage{amsmath}
\usepackage{amsfonts}
\usepackage{amssymb}
\usepackage{bbm}
\usepackage{bbold}
\usepackage{amsthm}
\usepackage{soul}
\usepackage{color}
\usepackage{cancel}

\usepackage{thmtools}
\usepackage{mathrsfs}

\theoremstyle{plain}
\newtheorem{theorem}{Theorem} 
\newtheorem{lemma}[theorem]{Lemma}

\newtheorem{definition}[theorem]{Definition}

\newtheorem{example}[theorem]{Example}

\journal{Journal of Mathematical Psychology}

\biboptions{authoryear}

\begin{document}

\begin{frontmatter}

\title{A Representation Theorem for Finite Best-Worst Random Utility Models}

\author{Hans Colonius}
\address{University of Oldenburg}

\ead[url]{https://uol.de/en/hans-colonius/}

\ead{hans.colonius@uni-oldenburg.de}


\begin{abstract}
This paper investigates best-worst choice probabilities (picking the best and the worst alternative from an offered set). It is shown that non-negativity of best-worst Block-Marschak polynomials is necessary and sufficient  for the existence of a random utility representation. The representation theorem is obtained by extending proof techniques employed by Falmagne (1978) for a corresponding result on best choices (picking the best alternative from an offered set). 
\end{abstract}


\end{frontmatter}


%
\section{Introduction}
Choosing an element from an offered set of alternatives  is arguably the most basic paradigm of preference behavior. Typically, if the same set is offered several times, the participant's choice will not always be the same. This is often attributed to the participant's preference fluctuating over time due to the effect of various alternatives to be compared, or to the difficulty of distinguishing between similar alternatives. Theories of choice behavior  try to account for the probability $P_A(a)$ of choosing an alternative $a$, say, from an offered set $B$, which is a subset of the base set $A$. This intrinsic randomness leads naturally to postulating the existence of a random variable $U_b$, say, for each alternative $b \in B$ representing the momentary strength of preference for alternative $a$. The participant is supposed to choose $a$ from $B$ if the momentary (sampled) value of $U_a$ exceeds that of any other alternative  $U_b$, $b \in B\setminus \{a\}$. Such a  \textit{random utility representation} can be traced back to the beginnings of psychophysics (\cite{fechner1860,wundt1978,thurstone1927}; see \cite{falmagne1985}, \cite{link1994}, and \cite{dzhafarov2011}) and may perhaps be considered a cornerstone of both early and contemporary theories of choice and decision making (``discrete choice'') in psychology, economics, statistics, and beyond (\cite{luce1959,block1960,tversky1972,corbinmarley1974, manskimcfadden1981,fishburn1998, louviere2000,hess2012}). 

An alternative to the best-choice paradigm is the \textit{ranking paradigm}: the participant is asked to rank-order all elements of an offered set from best to worst resulting in a probability distribution over all possible rankings. Many statistical models have been proposed for this paradigm (e.g. \cite{critchlow1991}) and some are directly connected to models of best choice. For example, a classic result of \cite{block1960} shows that, under specified conditions, the existence of a probability distribution over all possible rankings is necessary and sufficient for a random utility representation of best choices. The present paper concerns a relatively recent choice paradigm that, in terms of complexity, lies somewhat in-between the paradigms of choice and ranking.

\cite{marley1968} developed the \textit{reversible ranking} model where a ranking is obtained by a sequence of best and/or worst choices (which Marley called superior and/or inferior). Motivated, in part, by their familiarity with that work, \cite{finn1992} proposed a discrete choice task in which a participant is asked to select both the best and the worst option in an available (sub)set of options. For an offered set $B$, subset of a base set $A$, let $BW_B(a,b)$ be the probability that a participant chooses $a$ as best and $b$ as worst alternative in the set $B$. As observed in \cite{marley2005}, if there are, e.g., 4 items in $A$, one obtains information about the best option in 9 out of 11 possible non-empty, non-singleton subsets of $A$. Thus, best-worst choices contain a great deal of information about the person's ranking of options. Applications of this best-worst choice paradigm have strongly increased over the years. In particular, best-worst scaling is being used as a method of collecting ranking data which is then modeled in various ways related to the multinomial logit for best choices or to weighted versions of the rank ordered logit for repeated best choices (for details, see the monograph by \cite{louviere2015}, p.11pp.)
 
Marley and colleagues have developed various random ranking and random utility models for the best-worst paradigm (\cite{marley2005}; see also \cite{marleyregenwetter2017}). However, one problem has apparently remained unsolved up to now: What are necessary and sufficient conditions on the probabilities $BW_B(a,b)$ for the existence of a random utility representation, that is, for the existence of random variables $U_a, U_b, U_c \in A$ such that for any $B \subseteq A$  
\[BW_B(a,b)= P\left(\bigcap_{c\in B\setminus\{a,b\}} \{U_a \ge U_c \ge U_b\}\right) ?\]
The aim of this paper is to give a complete answer to this question. The solution leans heavily on an approach to the analogous problem for the choice paradigm developed by \cite{falmagne1978}. He was able to show that the non-negativity of certain linear combinations of choice probabilities (the so-called \textit{Block-Marschak polynomials}) is both necessary and sufficient for a random utility representation of best choices. An important part of Falmagne's ingenious proof is the construction of a probability measure on the set of rankings which, via the above-mentioned result by \cite{block1960}, implies the existence of a random utility representation. As it turns out, Falmagne's approach can be extended in a certain way to find necessary and sufficient conditions for the existence of a random utility representation for best-worst choices as well. 

This paper is organized as follows. After introducing some basic notation, we give a formal definition of a system of best-worst choice probabilities and its corresponding random utility representation. As a first result, a necessary condition for this representation in the form of linear inequalities of best-worst choices is provided, closely following arguments from \cite{falmagne1978} for best choices. Section~\ref{sec:BWBM} introduces best-worst Block-Marschak polynomials, shows how to recover best-worst choice probabilities from them using the Moebius inversion (Theorem~\ref{moebius}), and states the main representation theorem. Section~\ref{counting} investigates the structure of rankings (permutations) compatible with best-worst choices including some counting results. Section~\ref{bmrep}  contains the best-worst probability version of the Block-Marschak result, that is, a probability measure on rankings is necessary and sufficient for best-worst random utility representations. A probability measure on rankings is developed in the subsequent section such that best-worst choice probabilities are defined in terms of the probability measure on appropriate subsets of rankings, completing the proof of the main theorem. Finally, some open ends and related findings are outlined in the concluding section. Because some parts of our results have been obtained before, this will be mentioned throughout the text, to the best of our knowledge.
\section{Some definitions and basic results}
For a finite set $X$, we write $|X|$ for the number of elements in $X$, $\mathscr{P}(X)$ for the power set of $X$, $\mathscr{P}(X,i)$ for the set of all subsets of $X$ containing exactly $i$ elements. For  any nonempty set $X$, finite or not, let  $\Phi(X)$ be the set of all finite, nonempty subsets of $X$. 

%

\begin{definition}\normalfont{
		Let $A$ be a nonempty set  of $n$ elements ($n\ge 2$). For any $B\subseteq A$, with  $|B| \ge 2$, and any $a, b \in B$, $a\ne b$, 
		let $BW_B(a,b)\mapsto [0,1]$ denote  the probability that $a$ and $b$ are respectively chosen as best and worst elements in the subset $B$ of $A$. Let $\mathbb{BWP}= \{BW_B\,|\, B\subseteq A, BW_B(a,b)\mapsto [0,1],  a,b \in B, a\ne b\}$ be the collection of all those probabilities. Suppose that
		\[ \sum_{a,b \,\in B, \, a\ne b}  BW_B(a,b)=1 \hspace{1cm} (BW \in \mathbb{BWP})\]
		Then $(A,\mathbb{BWP})$ is called a \emph{system of best-worst choice probabilities}, or more briefly, a \emph{system}. If $A$ is finite, then $(A,\mathbb{BWP})$ is called \emph{finite}.}
\end{definition}
The quantities $BW_B(a,b)$ will be referred to as \emph{best-worst choice probabilities} indicating that an alternative $a$ is judged as best and $b$ as worst, when subset $B$ is available. 
\begin{definition}\label{RR}
\normalfont{A finite system $(A,\mathbb{BWP})$ is called \emph{a best-worst random utility system} if there exists a collection of  jointly distributed random variables $\{U_c\,|\, c \in A\}$ such that for all $B\subseteq A$ and $a,b\in B$, $a\ne b$,
%
%
%
\[BW_B(a,b)= P\left(\bigcap_{c\in B\setminus\{a,b\}} \{U_a \ge U_c \ge U_b\}\right) .\]
The collection $\{U_c\,|\, c \in A\}$ will be \emph{called a random representation of} $(A,\mathbb{BWP})$.}	
\end{definition}	

From now on, we will always assume the systems to be finite without mentioning it specifically.
Let us introduce two abbreviations:
\[ \mathrm{M}_B^+ =\max\{U_c\,|\, c\in B\} \;\;\text{ and } \;\; \mathrm{M}_B^- =\min\{U_c\,|\, c\in B\}. \]
Assume that $\{U_c\,|\, c\in A\}$ is a random representation of $(A,\mathbb{BWP})$. A simple, but important implication of Definition~\ref{RR} is the following: for any distinct $a,b\in A$
\begin{align*}
1 &= BW_{\{a,b\}}(a,b)+ BW_{\{a,b\}}(b,a)\\
& = P(U_a>U_b)+P(U_a=U_b)+P(U_b>U_a) +P(U_b=U_a).
\end{align*}
Since we also have 
\begin{align*}
1 & = P(U_a>U_b)+P(U_a=U_b)+P(U_b>U_a),
\end{align*}
we obtain 
\[ P(U_a=U_b)=0. \]
We conclude that with $a,b \in B$
\[ BW_B(a,b)=P(U_a\ge \mathrm{M}_B^+  \ge \mathrm{M}_B^- \ge U_b) = P(U_a > \mathrm{M}^+_{B\setminus\{a,b\}}  \ge \mathrm{M}_{B\setminus\{a,b\}}^- > U_b). \]

Now let $B_0, B_1, \ldots, B_n \in \Phi(A)$; for $a, b\in B_0$, 
\[ BW_{B_0}(a,b)= P(U_a\ge \mathrm{M}^+_{B_0} \ge \mathrm{M}_{B_0}^- \ge U_b) . \]
Denote 
\[ E_i=\{U_a\ge M^+_{B_i} \ge M_{B_i}^-\ge U_b\} \]
for $0\le i\le n$, and observe that, for $0\le i,j\le n$,.
\begin{align*}
E_i \cap E_j &=\{U_a\ge M^+_{B_i} \ge M_{B_i}^-\ge U_b \} \cap \{U_a\ge M^+_{B_j} \ge M_{B_j}^-\ge U_b \}\\
					&= \{U_a\ge M^+_{B_i \cup B_j} \ge M_{B_i \cup B_j}^-  \ge U_b\}.
\end{align*}
For $a, b\in B_0$, this implies 
\begin{equation*}
BW_{B_0\cup B_1 \cup \cdots B_n}(a,b)= P\left(\bigcap_{i=0}^n E_i\right).
\end{equation*}
From this follows, for example,
\begin{equation}\label{eq:1}
BW_{B_0}(a,b)-[BW_{B_0\cup B_1}(a,b)+BW_{B_0\cup B_2}(a,b)] + BW_{B_0\cup B_1 \cup B_2}(a,b)\ge 0,
\end{equation}
since 
\[ P(E_0)-[PE_0\cap E_1) +P(E_0 \cap E_2)] +P(E_0\cap E_1\cap E_2) \ge 0, \]
holds for arbitrary events $E_0, E_1, E_2$ in a probability space. The following theorem, and its proof, is completely parallel to the one in Falmagne (1978, Theorem 1).
\begin{theorem}\label{theor 2}
Let $(A,\mathbb{BWP})$ be a best-worst random utility system. For any $a,b \in B_0 \in \Phi(A)$, and any finite collection $\mathscr{B} =\{B_j\, |\, j\in J, B_j \subset A \text{ or }  B_j=\emptyset\}$, we have 
\begin{equation}\label{theor:2.1}
\sum_{i=0}^{|J|} (-1)^i \sum_{C\in \mathscr{P}(J,i)} BW_{B_0\cup \mathscr{B}(C)}(a,b) \ge 0,
\end{equation}
where $\mathscr{B}(C)=\bigcup_{j\in C} B_j$.
\end{theorem}
Note that the case $\mathscr{B}=\{B_1,B_2\}$ corresponds to Equation~\ref{eq:1}, while $\mathscr{B}=\{\emptyset\}$, $\mathscr{B}=\{B_1\}$ yield, respectively, 
\begin{align*}
BW_{B_0}(a,b) &\ge 0,\\
BW_{B_0}(a,b) -BW_{B_0 \cup B_1}(a,b)&\ge 0.
\end{align*}

\begin{proof}
Writing $E_0$ for the event $\{U_a \ge U_c\ge U_b, \, c\in B_0\setminus\{a,b\})\}$ and $E_j, j\in J,$ for the event $\{U_a\ge M^+_{B_j} \ge M_{B_j}^-\ge U_b\}$, we get
\[ BW_{B_0\cup \mathscr{B}(C)}(a,b) =P\left[\,\bigcap_{j\in C} (E_0\cap E_j)\right]. \]
The theorem follows from the fact that, for any finite collection$\{E_j\,|\,j\in J\}$ of events and any event $E_0$ in a probability space, we have
\begin{equation}\label{eq:3}
\sum_{i=0}^{|J|} (-1)^i \sum_{C\in \mathscr{P}(J,i)} P\left[\,\bigcap_{j\in C} (E_0\cap E_j)\right] \ge 0.
\end{equation}
Indeed, (\ref{eq:3}) certainly holds if $P(E_0)=0$; while if $P(E_0)\ne 0$, dividing on both sides by $P(E_0)$, (\ref{eq:3}) is equivalent, by \textit{Poincar\'{e}'s identity},   to 
\[ 1 \ge P\left(\bigcup_{j\in J} E_j \,|\, E_0\right) .\]
\end{proof}	
\section{Best-Worst Block-Marschak polynomials: the main theorem}\label{sec:BWBM}
Suppose $(A,\mathbb{BWP})$ is a system. Consider the following expressions:
\begin{align*}
& BW_A(a,b),\\
& BW_{A\setminus\{c\}}(a,b)- BW_A(a,b),\\
& BW_{A\setminus\{c,d\}}(a,b)-[BW_{A\setminus\{c\}}(a,b)+BW_{A\setminus\{d\}}(a,b)]+BW_A(a,b),\\
&  BW_{A\setminus\{c,d,e\}}(a,b)-[BW_{A\setminus\{c,d\}}(a,b)+BW_{A\setminus\{c,e\}}(a,b)+BW_{A\setminus\{d,e\}}(a,b)]+\\
&\hspace{2cm}[BW_{A\setminus\{c\}}(a,b)+BW_{A\setminus\{d\}}(a,b)+BW_{A\setminus\{e\}}(a,b)]- BW_A(a,b),\\
& \text{etc.}
\end{align*}
Each of these expressions is a case of the one in the left member of (\ref{theor:2.1}). In analogy to Falmagne's (1978) terminology, we introduce a compact notation.
\begin{definition}
	For any $B\subset A$, $B\neq A$, $a,b \in A\setminus B, a\ne b,$ in a system $(A,\mathbb{BWP})$, we define
\begin{equation}\label{BM}
K_{ab,B} = \sum_{i=0}^{|B|} (-1)^i \sum_{C\in \mathscr{P}(B, |B|-i)} BW_{A\setminus C}(a,b).
\end{equation}
The $K_{ab,B}$ are called \emph{best-worst Block-Marschak polynomials of} $(A,\mathbb{BWP})$, or \emph{best-worst BM polynomials}, for short.
\end{definition}
Observe that 
\begin{align*}
K_{ab,\emptyset} &=\sum_{i=0}^0 (-1)^i \sum_{C\in \mathscr{P}(\emptyset, 0-i)} BW_{A\setminus C}(a,b)\\
                             & = BW_A(a,b);\\
      K_{ab,\{c\}}    &=    BW_{A\setminus\{c\}}(a,b)- BW_A(a,b)  \\
                             & =    BW_{A\setminus\{c\}}(a,b)- K_{ab,\emptyset};\\
          K_{ab,\{c,d\}} & =  BW_{A\setminus\{c,d\}}(a,b)-[BW_{A\setminus\{c\}}(a,b)+BW_{A\setminus\{d\}}(a,b)]+BW_A(a,b) \\
          					&= BW_{A\setminus\{c,d\}}(a,b)- K_{ab,\{c\}} - K_{ab,\{d\}} -K_{ab,\emptyset} .
\end{align*}
Similar computations show that 
\begin{align*}
K_{ab,\{c,d,e\}} &=  BW_{A\setminus\{c,d,e\}}(a,b)-  K_{ab,\{c,d\}} -  K_{ab,\{c,e\}} -  K_{ab,\{d,e\}} \\
 							&  \hspace{1cm}-K_{ab,\{c\}} - K_{ab,\{d\}}  - K_{ab,\{e\}}  -K_{ab,\emptyset}
\end{align*}
These examples suggest the following result.
\begin{theorem}\label{moebius}
	Let $(A,\mathbb{BWP})$ be  a system of best-worst choice probabilities. Then, for all $B\subset A, B\ne A$, and $a,b \in A\setminus B$,
\begin{equation}
BW_{A\setminus B}(a,b) = \sum_{C\in \mathscr{P}(B)} K_{ab, C}.
\end{equation}
\end{theorem}                                                                                                                                                         %
Proof of this theorem is omitted here since it is completely analogous to the one in Falmagne (1978, Theorem 2, pp. 57--8) by replacing the ``ordinary'' BM polynomials by best-worst BM polynomials. Alternatively, with the same polynomial replacement, it is also analogous to the one given in \cite{colonius1984}, pp. 58-60, using \emph{M\"{o}bius inversion} (for the latter definition see, e.g. \cite{vanlint2001}).
%
%
\noindent We can now state the main theorem of this paper.
\begin{theorem}\label{main}
	A finite system of best-worst choice probabilities is a best-worst random utility system if and only if the best-worst Block-Marschak polynomials are nonnegative. 
\end{theorem}
The necessity follows from Theorem~\ref{theor 2} and the definition of best-worst Block-Marschak polynomials (Section~\ref{sec:BWBM}). 
The rest of the paper concerns the sufficiency, i.e. to show that if  \[ K_{ab,B}\ge 0 \] for all $B\subset A, B\ne A$ and  $a,b \in A \setminus B$, then $(A,\mathbb{BWP})$ is a best-worst random utility system. The proof requires an analysis of the Boolean algebra of the sets of permutations on $A$ in a system $(A,\mathbb{BWP})$ of best-worst choice probabilities.


%
\section{Sets of rankings and counting results}\label{counting}
The next definition is our basic tool for constructing the probability measure on the rankings (permutations) in the subsequent section. It is illustrated by a number of examples. Moreover, a counting lemma and a partitioning lemma needed for the construction are presented here.

For any $B\subseteq A$, we write $\Pi_B$ for the set of $|B|!$ permutations on $B$. For simplicity, we abbreviate $\Pi_A =\Pi$. Let $\ge$ be an arbitrarily chosen simple order on $A$. As usual, we write, for any $a,b \in A, a<b$ iff not $a\ge b$, and $a>b$ iff not $b \ge a$. 
\begin{definition}\label{rankingsets}\normalfont{
	Let $|A|=n$ with $n>\nolinebreak2$ and $B\subseteq A$  with $|B|=n-m$  $(n\ge m)$; for $B \in \Phi(A)$ and distinct $b_1, b_2, \ldots, b_{k^\prime}, b_{k^\prime+1}, b_{k^\prime+2},\ldots, b_k \in B$ $(1\le k^\prime < k\le n-m)$, define
	\begin{align*}
	& \mathrm{S}(b_1b_2\ldots b_{k^\prime};B;b_{k^\prime+1}b_{k^\prime+2}\ldots b_k)= \nonumber\\
	&\quad \{\pi \in \Pi\,|\,    \pi(b_1)>\pi(b_2)>\cdots  >\pi(b_{k^\prime})>\pi(b) >\pi(b_{k^\prime+1})  >\pi(b_{k^\prime+2})>\cdots 
	>\pi(b_k),\nonumber \\& \text{ for all $b\in B\setminus\{b_1, \ldots, b_k\}$}  \}.
	\end{align*}}
\end{definition}
For simplicity, we write $b_1b_2\ldots b_k$ for the ranking (defined by $>$) corresponding to permutation $\pi$ with $  \pi(b_1)>\pi(b_2)>\cdots  >\pi(b_k)$. Moreover, if no confusion arises we also omit the semicolons around set $B$ and write  \\$\mathrm{S}(b_1b_2\ldots b_{k^\prime}Bb_{k^\prime+1}b_{k^\prime+2}\ldots b_k)$. The following examples illustrate the properties of the sets defined above.
\begin{example}[label=exa1:cont]
	Let $A=\{p,q,r,u,v\}$ and  $B=\{p,q,r\}$; determine $\mathrm{S}(pBq)$. 
	Note that here, $r$ must always be between $p$ and $q$; we set up a table such that the elements $u$ and $v$, which are not in $B$, are positioned among the elements of $B$ in both possible orders; for order $uv$  this yields Table~\ref{tab1}.
	\begin{table}\label{tab1}
		\begin{center}	
			$	\begin{array}{|c|c|c|c|c|c|c||c|} 
			\hline 	
			\mathtt{P}_1	& p & \mathtt{P}_2  & r & \mathtt{P}_3 & q& \mathtt{P}_4 &\text{ranking}\\ 	\hline 	\hline
			uv	& &  &  &  & &  & uvprq\\ 
			u	&  &v  &  &  &  & &upvrq \\ 
			u	&  &  &  & v &  &  &uprvq\\ 
			u	&  &  &  &  &  & v & uprqv\\ 	\hline 	\hline
			&  & uv &  &  &  &  & puvrq\\ 
			&  &u  &  & v &  & &purvq \\ 
			&  & u &  &  &  &v  & purqv\\ 	\hline 	\hline
			&  &  &  & uv &  &  & pruvq\\ 
			&  &  &  &  u&  &  v &pruqv\\ 	\hline 	\hline
			& & & & & &uv &prquv\\	\hline
			\end{array}$\caption{for Example~\ref{tab1}}
		\end{center}
	\end{table}	
	For order $vu$ an analogous table exists. Thus, the total number of rankings is $|\mathrm{S}(pBq)|=20$. Moreover, 
	\[  \mathrm{S}(prBq)=\mathrm{S}(pBrq)=\mathrm{S}(pBq). \] 
\end{example}
\begin{example}[label=exa2:cont] Let $A=\{p,q,r,u,v\}$ and  $B=\{p,q\}$ then
	\[  \mathrm{S}(pBq)= \{\pi \in \Pi_A\,|\, \pi(p)>\pi(q)\}.\] 
	With $|\Pi_A|=5!=120$, it follows that $|\mathrm{S}(pBq)|=60$  since exactly half of the permutations have $p$ ranked before $q$.
\end{example}
\begin{example}[label=exa3:cont]
	Let $|A|=n$ and $B=\{b\}$ for $a,c \in A\setminus B$
	\[ \mathrm{S}(aA\setminus Bc) = \mathrm{S}(baAc) + \mathrm{S}(aAcb)+ \mathrm{S}(aAc)\]
\end{example}
\begin{example}
	Let $|A|=n$ and $B=\{b,d\}$; for $a,c \in A\setminus B$
	\begin{align*} &\mathrm{S}(aA\setminus Bc) = \mathrm{S}(baAc)+\mathrm{S}(aAcb) +\mathrm{S}(daAc)+\mathrm{S}(aAcd)    + \mathrm{S}(bdaAc)    \\
	&\qquad \quad +\mathrm{S}(dbaAc)   +\mathrm{S}(aAcbd)  +\mathrm{S}(aAcdb) + \mathrm{S}(baAcd)  +\mathrm{S}(daAcb)+ \mathrm{S}(aAc).\end{align*}
\end{example}
Counting the number of rankings  in certain sets gives some insight and, in any case,  is useful for checking results. For $|A|=n$ with $n>2$ and $B\subseteq A$  with $|B|=n-m$ we want to determine the number of elements contained in $\mathrm{S}(b_1b_2\ldots b_{k^\prime}Bb_{k^\prime+1}b_{k^\prime+2}\ldots b_k)$, where $0\le k^\prime < k\le n-m$. Before presenting the general result in the next lemma we consider an example from above.
\begin{example}[continues=exa1:cont]
	For $A=\{p,q,r,u,v\}$ and  $B=\{p,q,r\}$ determine $|\mathrm{S}(pBq)|$. Here, $|A|=n=5$, $|B|=n-m=3$, and Table~\ref{tab1} lists all 10 possible rankings of $A$ where $u$ is ranked before $v$. Specifically, there are  4 possible positions for $u$, denoted $\mathtt{P}_1, \mathtt{P}_2 , \mathtt{P}_3,\mathtt{P}_4$ in the top row of the table, and once a position for $u$ is chosen, $v$ can only positioned ``somewhere to the right'' of $u$, resulting in a total of 10 rankings. Choosing $v$ first and $u$ second yields another 10 rankings for a total of 20.
\end{example}
\begin{lemma}\label{count:lemma}
	Let $|A|=n$ with $n>2$ and $B\subseteq A$  with $|B|=n-m$; for $1\le k^\prime < k\le n-m$ the number of elements contained in \[\mathrm{S}(b_1b_2\ldots b_{k^\prime}B\,b_{k{^\prime +1}}b_{k{^\prime +2}}\ldots b_k)\]equals\footnote{for $m=0$, the $\Pi$ term is set to 1.},
	\begin{align*}
	|\mathrm{S}(b_1b_2\ldots b_{k^\prime}Bb_{k^\prime+1}b_{k^\prime+2}\ldots b_k)|= &(n-m-k)!  \prod\limits_{i=1}^{m} [n-(m-i)]\\
	=& \frac{(n-m-k)! \;n!}{(n-m)!}.														
	\end{align*}
\end{lemma}
%
Note that the result does not depend on $k^\prime$. A proof is in \ref{A1}. We consider a few of the above examples for illustration. 
\begin{example}[continues=exa2:cont]
	With $A=\{p,q,r,u,v\}$ and  $B=\{p,q\}$, we have $n=5$, $m=3$, and $k=2$ for $\mathrm{S}(pBq)$. Thus, by Lemma~\ref{count:lemma}
	\[ | \mathrm{S}(pBq)|= (5-3-2)! (5-2) (5-1) 5 = 1! 60 =60,\]
	as inferred before by a different argument.
\end{example}
\noindent Finally, consider Example~\ref{exa3:cont}:
\begin{example}[continues=exa3:cont]
	With $|A|=n$ and $B=\{b\}$, we need to show that, for $a,c \in A\setminus B$,
	\[ |\mathrm{S}(aA\setminus Bc)| = |\mathrm{S}(baAc)| + |\mathrm{S}(aAcb)|+ |\mathrm{S}(aAc)|,\]
	because all sets on the right are pairwise disjoint. For the left hand side, $m=1$ and $k=2$, thus 
	\[  |\mathrm{S}(aA\setminus Bc)| = (n-1-2)! (n-(1-1)) = (n-3)! \, n.\]
	For the first two sets on the right hand side, $m=0$ and $k=3$,thus
	\[ |\mathrm{S}(baAc)| = |\mathrm{S}(aAcb)| = (n-0-3)! \, 1=(n-3)!; \]
	and, for the third set $m=0$ and $k=2$, thus
	\[  |\mathrm{S}(aAc)| = (n-0-2)! \, 1 = (n-2)!;\]
	summing up the numbers from the right, $2 (n-3)! + (n-2)!= (n-3)! \,[2+n-2]=(n-3)!\,n,$ which equals the number on the left.
\end{example}
\begin{lemma}\label{lemma:disjoint}	
	For any $B\subset \Phi(A)$ and $a,b \in A\setminus B$
	\begin{equation}
	S(aA\setminus Bb) = \sum_{C\in \mathscr{P}(B)}\sum_{\pi\in \Pi_C} \sum_{(\pi_1 \pi_2)=\pi} S(\pi_1 aAb \pi_2);
	\end{equation}
	here, the last summation over $(\pi_1 \pi_2)=\pi$ means that ranking $\pi$ is split into all possible pairs, for example, $cde$ is split into \[ (cde)(),(cd)(e),(c)(de),()(cde)\] so that $(cd)(e)$ corresponds to $S(cdaAbe)$, $()(cde)$ to $S(aA;bcde)$, etc.
\end{lemma}
The proof is in \ref{A2}.  The following example illustrates the lemma.
\begin{example}[continues=exa2:cont]
	With $A=\{p,q,r,u,v\}$ and  $B=\{p,q\}$, let us consider $S(uA\setminus Bv)$; thus, $n=5$, $|A\setminus B|=3=n-m$, so $m=2$, and $k=2$. From Lemma~\ref{count:lemma} ,
	\[ |S(uA\setminus Bv)|= \frac{(n-m-k)! \;n!}{(n-m)!}	=\frac{1!\;5!}{3!} = 20.\]
	These 20 rankings are listed in the first column of Table~\ref{tab2} partitioned into  the  additive components $S(\pi_1 uAv \pi_2)$ (right-hand side in Lemma~\ref{lemma:disjoint} ).  The second column shows the corresponding subsets $C$  of $B$. For instance, with $C=\emptyset$, $S(\pi_1 uAv \pi_2)=S(uAv)$ and $|S(uAv)|=6$; and with $C=\{p,q\}$, $\Pi_C=\{pq,qp\}$ and for each permutation there are 3 ways to split them into $(\pi_1,\pi_2)$: $(pq)(), ()(pq), (p)(q)$ and $(qp)(), ()(qp), (q)(p)$.
	\begin{table}\label{tab2}
		\begin{center}	
			$	\begin{array}{|c|c|c|c|c} 
			\hline 	
			S(uA\setminus Bv)&S(\pi_1 uAv \pi_2)&|S(\pi_1 uAv \pi_2)| \\ 	\hline 	\hline
			upqrv& & n=5; m=0;k=2 \\ 
			uprqv &  & \frac{(n-m-k)! \;n!}{(n-m)!}	=6\\ 
			uqprv& C=\emptyset &\\ 
			uqrpv & &\\ 
			urpqv & &\\
			urqpv & &\\ \hline 
			puqrv& &  n=5; m=0; k=3\\
			purqv& C=\{p\} & \frac{(n-m-k)! \;n!}{(n-m)!}	=2\\ 
			uqrvp&  & \\
			urqvp& &  \\ \hline
			quprv & & n=5; m=0; k=3\\
			qurpv& C=\{q\} & \frac{(n-m-k)! \;n!}{(n-m)!}	=2\\
			uprvq& &\\
			urpvq& &\\ \hline
			pqurv& &  n=5; m=0; k=4\\
			qpurv&  &  \frac{(n-m-k)! \;n!}{(n-m)!}	=1\\
			urvpq&C=\{p,q\} &\\
			urvqp& & \\
			purvq& &\\
			qurvp& & \\ \hline
			\end{array}$\caption{Example~\ref{tab2} for explanation, see text.}
		\end{center}
	\end{table}	
\end{example}
\begin{lemma}\label{lemma:partition_kprime}
	For all $1 \le k^\prime < k \le n$,
\begin{align*}
\sum_{a \in A\setminus\{a_1,\ldots, a_k\}}S(a_1,\ldots a_{k\prime}\,aA\,a_{k^\prime +1}\ldots a_{k}) & &\\ 
& \hspace{-2cm}=\sum_{a \in A\setminus\{a_1,\ldots, a_k\}}S(a_1,\ldots a_{k^\prime}\,A\,a\,a_{k^\prime +1}\ldots a_{k}) \\
&\hspace{-2cm}= S(a_1,\ldots a_{k^\prime}\,A\,a_{k^\prime +1}\ldots a_{k})
\end{align*}
\end{lemma}
\noindent The straightforward proof is omitted. Next, we define the union of certain sets of rankings which the probability measure will ultimately be constructed on.
For $ k \ge2$ and distinct $a_1,a_2,\ldots,a_k \in A,$
\begin{equation}\label{SA}
S_A(a_1,\ldots,a_k)= \sum_{k^\prime =1}^{k-1} S(a_1\ldots a_{k^\prime}\,A\,a_{k^\prime +1}\ldots a_k).
\end{equation}
\section{A Block-Marschak type lemma}\label{bmrep}
\noindent This lemma provides a critical step in our proof. It is a variant of the well-known result by Block \& Marschak (1960) (see also Marley \& Louviere 2005, section on random ranking models).
\begin{lemma}\label{lem:BM}
	$(A,\mathbb{BWP})$ is a best-worst random utility system if and only if there exists 	a probability measure $\mathrm{P}[\,.\,]$ on $\mathscr{P}(\Pi)$ satisfying
	\begin{equation}\label{extBM}
	BW_B(a,b) = \mathrm{P}[S(aBb)] 
	\end{equation}
	for all $B\in \Phi(A)$	and $a,b \in B$.
\end{lemma}
\begin{proof}  (Lemma~\ref{lem:BM})\\
	For simplicity, we set $A=\{1,2,\ldots,n\}$ and take $\ge$ as the natural order of the reals.
	
	(Necessity) Let $\{U_i \,|\, 1\le i \le n  \}$ be a random representation of $(A, \mathbb{BWP})$, with joint probability measure $\mathrm{P}$ satisfying~(\ref{extBM}). For any $\pi \in \Pi$, define
	\[ p(\{\pi\}) = \mathrm{P}[ U_{\pi^{-1}(1)} < U_{\pi^{-1}(2)} <\cdots < U_{\pi^{-1}(n)} ].\]
	It is easy to verify that $p$ is a probability distribution on $\Pi$ that can be extended to a probability measure $\mathrm{P}$ on $\mathscr{P}(\Pi)$, satisfying~(\ref{extBM}).
	
	(Sufficiency) Conversely, suppose that (\ref{extBM}) holds for some probability measure $\mathrm{P}$ on $\mathscr{P}(\Pi)$. Define the joint distribution of a collection $\{U_i \,|\, 1\le i \le n  \}$ of random variables by
	\begin{equation}\label{BMrev}
	\mathrm{P}[U_1=\xi_1, U_2=\xi_2,\ldots,U_n=\xi_n] =
	\begin{cases}
	\mathrm{P}(\pi)  & \text{ if $ \pi(i) = \xi_i$, $1\le i \le n$, }\\
	0 & \text{otherwise}	
	\end{cases}
	\end{equation}
	for all $n$-tuples $\xi_1, \xi_2, \ldots, \xi_n$ of real numbers. It can be checked that then 
	\begin{equation}\label{BMrev2}
	BW_B(i,j) = \mathrm{P}[U_i \ge M_B^+ \ge M_B^- \ge U_j]
	\end{equation}
	for any $B \in \Phi(A)$, $i,j \in B$.
\end{proof}

\section{Defining a probability measure on the rankings of $A$} \label{probonranking}
\noindent This section completes the sufficiency part of the main theorem (Theorem~\ref{main}). Thus, we assume the best-worst BM polynomials to be non-negative. The first step is to find a function on the sets \[S(a_1,\ldots,a_{k^\prime}Aa_{k^\prime +1},\ldots,a_k)\subset \Pi_A,\] with $|A|=n$ and $1\le k^\prime <k\le n$. To this end, a function $\mathrm{F}^\prime$ is defined inductively. For $k=2$ (thus, $k^\prime =1$), define 
\begin{align}
\mathrm{F}^\prime[S(a_1Aa_2)]:= K_{a_1 a_2, \emptyset}.
\end{align}
For $k\ge 3$, $k^\prime <k$, we define 
\begin{align}
\hspace{-0cm}\mathrm{F}^\prime[S(a_1\ldots a_{k^\prime}A\,a_{k^\prime +1} \ldots a_k)] &:=\nonumber\\ 
&\nonumber\\ 
&\hspace{-6cm}\frac{\mathrm{F}^\prime[S(a_1\ldots a_{k^\prime-1} A\, a_{k^\prime +2 }\ldots a_{k-1})] \; \;K_{a_{k^\prime} a_{k^\prime+1},  \{a_1,\ldots,a_{k^\prime-1},a_{k^\prime+2}, \ldots,a_{k-1}\}}/(n-k)!}
{\displaystyle\sum_{\pi \in  \Pi_{\{a_1,\ldots,a_{k^\prime-1},a_{k^\prime+2}, \ldots,a_{k-1}\}}} \mathrm{F}^\prime[S(\pi(a_1) \ldots \pi(a_{k^\prime -1}) A \,\pi(a_{k^\prime +2}) \ldots \pi(a_{k-1})]}
\end{align}
assuming the denominator $ >0$, and set $\mathrm{F}^\prime=0$ otherwise. 
\begin{lemma}\label{functionFprime} 
\begin{itemize}
	\item[(a)] $\mathrm{F}^\prime \ge 0$;
	\item[(b)] \[\sum_{\pi \in \Pi_B} \mathrm{F}^\prime[S(\pi_1 aAb\, \pi_2)]=K_{ab,B}/(n-k)!\] 
	for any $B \subset A$ and  $a,b \in A\setminus B$, and  $\pi$ of the form $(\pi_1 \pi_2)$.
\end{itemize}
\end{lemma}
\noindent  Non-negativity of $\mathrm{F}^\prime$  \textit{(a)} follows from assuming non-negative best-worst BM polynomials and \textit{(b)} is  immediate from the above recursive definition. Note that $(n-k)!$ is number of elements in $S(\pi_1 aAb\, \pi_2)$. This suggests the following interpretation of the  $K_{ab,B}$ as the probability measure of all rankings of $A$ with $a$ as best and $b$ as worst ignoring all alternatives that are in $B$, with a specific number of elements of $B$ above $a$ and below $b$ according to $(\pi_1 \pi_2)$.

Next, we extend $\mathrm{F}^\prime$ to a function $\mathrm{F}$ on the sets $S_A(a_1,\ldots,a_k)$ ($k\le n$) by defining:
\begin{align}
\mathrm{F}[S_A(a_1,\ldots,a_k)]&\equiv \mathrm{F}\left[ \sum_{k^\prime =1}^{k-1} S(a_1\ldots a_{k^\prime}\,A\,a_{k^\prime +1}\ldots a_k)\right]\nonumber \\
&:=  \sum_{k^\prime =1}^{k-1}  \mathrm{F}^\prime [S(a_1\ldots a_{k^\prime}\,A\,a_{k^\prime +1}\ldots a_k)]
\end{align}

\noindent For $\mathrm{F}$ to be a probability distribution on $ \mathscr{P}(\Pi)$, we need to show
\begin{itemize}
	\item[(i)] $\mathrm{F} \ge 0$;
	\item[(ii)] $\sum_{\pi \in \Pi} \mathrm{F}[S(\pi(a_1), \ldots, \pi(a_n)] = 1.$
\end{itemize}
Obviously, (i) follows from the non-negativity of $\mathrm{F}^\prime$. We obtain (ii) as a special case of the general result that for $2\le j\le n$  ($k^\prime < j$)
\begin{equation}\label{sumgen}
\sum_{C\in \mathscr{P}(A,j)} \sum_{\substack{\pi \in \Pi_C\\ C=\{a_1,\ldots,a_j\}}} \mathrm{F}[S_A(\pi(a_1), \ldots, \pi(a_j)] =1;
\end{equation}
(ii) is then obtained from (\ref{sumgen}) for $j=n$. Equation~\ref{sumgen} is proved by induction on $j$. For $j=2$, we have 
\begin{align*}
\sum_{C\in \mathscr{P}(A,2)} \sum_{\substack{\pi \in \Pi_C\\ C=\{a_i,a_\ell\}}} \mathrm{F}[S_A(\pi(a_i),\pi(a_\ell)]&= \sum_{\substack{a_i,a_\ell \in A\\ a_i \ne a_\ell}} \mathrm{F}[S_A(a_i,a_\ell)]\\
&  \hspace{-6cm}= \sum_{\substack{a_i,a_\ell \in A\\ a_i \ne a_\ell}} \mathrm{F}^\prime[S(a_i A a_\ell)]
  =\sum_{\substack{a_i,a_\ell \in A\\ a_i \ne a_\ell}}K_{a_i a_\ell, \emptyset}
  = \sum_{\substack{a_i,a_\ell \in A\\ a_i \ne a_\ell}} BW_A(a_i,a_\ell)   =1.
\end{align*}
Now assume that (\ref{sumgen}) holds for all $j$ with $2\le j\le k-1<n$  ($k^\prime < k$); then 
\begin{align*}
&\sum_{C\in \mathscr{P}(A,k)} \sum_{\substack{\pi \in \Pi_C\\ C=\{a_1,\ldots,a_k\}}} \mathrm{F}[S_A(\pi(a_1), \ldots,  \pi(a_k)] &\\
& =\sum_{C\in \mathscr{P}(A,k)} \sum_{\substack{\pi \in \Pi_C\\ C=\{a_1,\ldots,a_k\}}}\sum_{k^\prime =1}^{k-1}  \mathrm{F}^\prime [S(a_1\ldots a_{k^\prime}\,A\,a_{k^\prime +1}\ldots a_k)]\\
&=\sum_{C^\prime\in \mathscr{P}(A,k-1)} \sum_{\substack{\pi^\prime \in \Pi_C^\prime\\ C^\prime=\{a_1,\ldots,a_{k-1}\}}}\sum_{k^\prime =1}^{k-2} \;\sum_{a\in A\setminus C^\prime} \mathrm{F}^\prime[S(\pi^\prime(a_1) \ldots \pi^\prime(a_{k^\prime })\,a A \,\pi^\prime(a_{k^\prime +1 }) \ldots \pi^\prime(a_{k-1})]&\\
& =\sum_{C^\prime\in \mathscr{P}(A,k-1)} \sum_{\substack{\pi^\prime \in \Pi_{C^\prime}\\ C^\prime=\{a_1,\ldots,a_{k-1}\}}}\sum_{k^\prime =1}^{k-2}  \mathrm{F}^\prime [S(\pi^\prime(a_1)\ldots \pi^\prime(a_{k^\prime})\,A\,\pi^\prime(a_{k^\prime +1})\ldots \pi^\prime(a_{k-1})]\\
&=\sum_{C^\prime\in \mathscr{P}(A,k-1)} \sum_{\substack{\pi^\prime \in \Pi_C^\prime\\ C^\prime=\{a_1,\ldots,a_{k-1}\}}}\mathrm{F}[S(\pi^\prime(a_1), \ldots, \pi^\prime(a_{k-1})]\\
&=1
\end{align*}
by the induction hypothesis. Thus, (16) holds for $j=n$. We extend the probability distribution on $\Pi$ in a standard way to obtain a probability measure $\mathrm{P}$ on $\mathscr{P}(\Pi)$. 

In view of Lemma~\ref{lem:BM}, we need to show, finally,  that
\begin{equation}\label{fin}
	BW_B(a,b) = \mathrm{P}[S(aBb)] 
\end{equation}
for all $B\in \Phi(A)$	and $a,b \in B$. Now, 
\begin{align*}
\mathrm{P}[S(aBb)] &=\mathrm{P} \left[ \sum_{C\in \mathscr{P}(A\setminus B)} \sum_{\pi\in \Pi_C} \sum_{\pi =(\pi_1 \pi_2)} S(\pi_1 aA \,b\, \pi_2)\right] &&\text{by Lemma~\ref{lemma:disjoint}}\\
&=  \sum_{C\in \mathscr{P}(A\setminus B)} K_{ab, C} &&\text{by Lemma~\ref{functionFprime}}\\
&= BW_B(a,b) && \text{by Theorem~\ref{moebius}}
\end{align*}
completing the proof of Theorem~\ref{main}.

\section{Conclusion}
This paper adds to the theoretical underpinnings of the best-worst choice paradigm: non-negativity of certain linear combinations of best-worst choice probabilities (i.e. the best-worst Block-Marschak polynomials) is shown to be necessary and sufficient for a random utility representation of these choice probabilities. Most results on this paradigm, up to now, are contained in \cite{marley2005} relating models of best choices, worst choices, and best-worst choices, based on random ranking and random utility, to each other and pointing to open problems. Recently, \cite{depalma2017} presented additional relations between these paradigms under slightly stricter random utility representations and derived various expressions for independent and generalized extreme value distributed utilities.

As pointed out repeatedly, our results can, perhaps surprisingly, be considered a straightforward extension of Falmagne's work on representing best choices. In this context, it should be noted that \cite{fiorini2004} gave an alternative proof of Falmagne's result using polyhedral combinatorics. His proof is very short and elegant, reducing the representation theorem to a complete linear description of the \textit{multiple choice polytope}. In view of this, it seems obvious to look for an analogous description of the best-worst choice polytope\footnote{For definitions, we refer to the literature mentioned here.}, as has been undertaken in \cite{doignon2015}, but we are not aware of a solution of the representation problem using these techniques yet.

Given that non-negativity of the  Block-Marschak polynomials for best choices guarantees the existence of an underlying random utility (aka\textit{ random scale}), testing of this property has recently been in the focus of interesting work in signal detection theory for recognition memory (see, e.g. \cite{kellen2018}).  Thus, it would not be surprising to see analogous applications appear for best-worst Block-Marschak polynomials. 

Finally, \cite{falmagne1978} presented some results on the uniqueness of the random utility representation for best choices (see also \cite{colonius1984} for additional results). We leave it as an open problem to derive corresponding properties for the case of the best-worst random utility representation developed here.
\section*{Acknowledgment} I am grateful to Tony Marley, Adele Diederich, and Jean-Claude Falmagne for helpful comments. Any errors and deficiencies, of course, are the author's responsibility alone. This work was supported in part by DFG grant CO 94/6-1 to H. Colonius and Tony  Marley (U. of Victoria, Canada).
\newpage
\appendix
\section{Proof of  Lemma~\ref{count:lemma}}
\label{A1}
For a proof of Lemma~\ref{count:lemma}, we need another lemma.
\begin{lemma} For integer $m,n$ with $2\le m <n$,
	\label{lem:4}
	\begin{equation}\label{lemma4}
	\sum_{i_1=1}^{n-m+1}\,	\sum_{i_2=i_1}^{n-m+1} \cdots 	\sum_{i_{m-1}=i_{m-2}}^{n-m+1}	(n-m+2-i_{m-1}) = \frac{1}{m!} \,\prod_{i=1}^m (n-m+i)= {n\choose m}.
	\end{equation}
\end{lemma}
\begin{proof}
	Proof is by induction over $m$, $m<n$. For $m=2$, 
	\begin{align*}
	\sum_{i_1=1}^{n-1}\, (n-i_1)&=  	\sum_{i_1=1}^{n-1}\, n -\sum_{i_1=1}^{n-1}\, i_1\\
	&= n (n-1) -\frac{1}{2}n(n-1) \\
	&=\frac{1}{2}\, n\, (n-1),
	\end{align*}
	which is easily seen to be equal to the right-hand side of  (\ref{lemma4}) for $m=2$. Now let (\ref{lemma4}) be true for $m$; then, by straightforward but  tedious  algebra (omitted),
	\begin{align*}
	\sum_{i_1=1}^{n-m}\,	\sum_{i_2=i_1}^{n-m} \cdots 	\sum_{i_{m}= i_{m-1}}^{n-m}	(n-m+1-i_{m})&\\
	= \frac{n-m}{m+1}\;	\sum_{i_1=1}^{n-m+1}\,	\sum_{i_2=i_1}^{n-m+1} \cdots 	\sum_{i_{m-1}=i_{m-2}}^{n-m+1}	(n-m+2-i_{m-1}),
	\end{align*}
	which completes the induction step\footnote{An alternative proof without induction was suggested to me by Florian Hess (Oldenburg).}.
\end{proof}
\begin{proof}(Lemma~\ref{count:lemma})\\
	We have to find the number of rankings on $A$ compatible with the elements in $B$  satisfying the partial ranking $\pi(b_1)>\dots > \pi(b_k)$. For $m=0$, we have $A=B$ and 
	\[  |\mathrm{S}(b_1b_2\ldots b_{k^\prime}Bb_{k^\prime+1}b_{k^\prime+2}\ldots b_k)| = (n-k)!\]
	is just the number of permutations of the elements of $A$ with a fixed order of the $k$ elements. For $m=1$, a similar argument goes through.
	
	Note that, for the first of the $m$  $(m\ge 1)$ elements in $A\setminus B$, say $a_1$, there are $n-m+1$ possible positions relative to the $n-m$ elements in $B$ (see, e.g. Table\ref{tab1}).
	Let $i_1$, $1\le i_1\le n-m+1$ be the number of the position chosen for $a_1$,   $i_2$ the number of the position chosen for the second element $a_2$, etc. We can assume that $i_1\le i_2\le \cdots \le i_{m-n}$ in order to maintain the ranking. For example, 
	assuming $i_1=1$, then element $a_2$ also has $n-m+1$ possible positions; if $i_1=2$, then position 1 is no longer available for $a_2$ but all positions $\ge 2$ so that $\pi(a_1)>\pi(a_2)$ (see, e.g. Table\ref{tab1}).  In order to count the number of possibilities for the first two elements, consider the sum
	\begin{equation}
	\label{sum1}
	\sum_{i_1=1}^{n-m+1} \, (n - m + 2 - i_1) =\frac{1}{2} (n-m+1)(n-m+2)
	\end{equation}
	%
	Continuing this way for all $m$ elements results in 
	\begin{align*}
	&	\sum_{i_1=1}^{n-m+1}\,	\sum_{i_2=i_1}^{n-m+1} \cdots 	\sum_{i_{m-1} = i_{m-2}}^{n-m+1}	(n-m+2-i_{m-1})  \\
	&	  =\frac{1}{m!} \,\prod_{i=1}^m (n-m+i) ,
	\end{align*}
	with the equality following according to Lemma~\ref{lem:4}. Because we have considered a specific order for the $m$ elements of $B$, in order to obtain the total number of rankings, the above has to be multiplied by the number $m!$ of possible permutations. Moreover, the number of elements in $B$ that had not been ranked, amounts to $n-m-k$. Considering all possible orders of  these $n-m-k$, we need to also multiply by $(n-m-k)!$, yielding the lemma.
\end{proof}
\section{Proof of  Lemma~\ref{lemma:disjoint}}
\label{A2}
\begin{proof} 
	First, we show that the union in the right member of Eq.~\ref{lemma:disjoint} is disjoint.  For $(\pi_1,\pi_2)=\pi\ne \pi^\prime=(\pi_1^\prime,\pi_2^\prime)$, the sets $S(\pi_1 aAb \pi_2)$ and $S(\pi_1^\prime aAb \pi_2^\prime)$ are clearly disjoint, so that disjointness remains to be shown for the two first summation signs in (\ref{lemma:disjoint}).  For now, let us abbreviate $S(\pi_1 aAb \pi_2)$  as $S_\pi(aAb)$. For $|B|=1$, the lemma's claim is implicit in Example~\ref{exa3:cont} , and for  $|B|=2$, Example~\ref{exa2:cont} demonstrates the partition. 
	
	Let $|B|\ge 3$ take $C,C\,^\prime \in  \mathscr{P}(B)$, $\pi \in \Pi_C, \pi^\prime \in \Pi_{C^\prime}$, $\pi \ne \pi^\prime$. Suppose $C=C\,^\prime$. Then $|C|\ge 2$ (otherwise,  $\Pi_C=\{\pi\}$, contradicting $\pi \ne \pi^\prime$), and there will be at least two elements $d,e \in C$ such that such that $\xi(d)<\xi(e)$ for all $\xi \in S_\pi(aAb)$, while $\xi^\prime(e)<\xi^\prime(d)$ for all $\xi^\prime \in S_{\pi^\prime}(aAb)$. Thus
	\begin{equation}\label{cap}
	S_\pi(aAb)\cap S_{\pi^\prime}(aAb)=\emptyset.
	\end{equation}
	Then case $C\ne C\,^\prime$ is similar. For example, suppose $d\in C\setminus C\,^\prime$, then $\xi(d)>\xi(a)$ or $\xi(b)>\xi(d)$ for all $\xi \in S_\pi(aAb)$, while $\xi^\prime(a)>\xi^\prime(d) >\xi^\prime(b)$  for all $\xi^\prime \in S_{\pi^\prime}(aAb)$, entailing again Eq.~\ref{cap} .
	
	We turn to the proof of equality, and write $G(a,b,A,B)$ for the right member of \ref{lemma:disjoint} . Assume $\xi \in S(aA\setminus Bb)$. Then, either $\xi(a) > \xi(c)>\xi(b)$ for all $c \in B$, implying $\xi \in S(aAb)\subset G(a,b,A,B)$; or, there are $c_1,c_2,\ldots ,c_j \in B$ such that \[\xi(c_1)>\xi(c_2)> \ldots > \xi(c_{j_1})>\xi(a)>\xi(b)>\xi(c_{j_1+1})>\ldots >\xi(c_j).\] This yields $\xi \in S(c_1c_2\ldots c_{j_1}\,aAb\,c_{j_1+1}\ldots c_j)\subset G(a,b,A,B).$ We conclude that $S(aA\setminus Bb)\subset G(a,b,A,B)$.The converse implication follows from the fact that for any choice of $C\in \mathscr{P}(B)$ and $\pi \in \Pi_C$, we have $S_\pi(aAb) \subset S(aA\setminus Bb)$.
\end{proof}
\newpage
\bibliography{best-worst}
\end{document}